\documentclass[11pt]{amsart}

\usepackage{amssymb}
\usepackage{enumerate}
\usepackage{eucal}
\usepackage[margin=1in]{geometry}
\usepackage[colorlinks=true, citecolor=blue, linkcolor=blue, urlcolor=black]{hyperref}

\newtheorem{theorem}{Theorem}[section]
\newtheorem{corollary}[theorem]{Corollary}
\newtheorem{lemma}[theorem]{Lemma}
\newtheorem{proposition}[theorem]{Proposition}

\theoremstyle{definition}
\newtheorem{definition}[theorem]{Definition}

\newtheorem{remark}[theorem]{Remark}

\numberwithin{equation}{section}

\DeclareMathOperator{\Aut}{Aut}
\DeclareMathOperator{\Ext}{Ext}
\DeclareMathOperator{\h}{H}
\DeclareMathOperator{\Hom}{Hom}
\DeclareMathOperator{\sym}{Sym}
\DeclareMathOperator{\iso}{Iso}
\newcommand{\bb}[1]{\mathbb{#1}}
\newcommand{\bbk}{\Bbbk}
\newcommand{\inj}{\hookrightarrow}
\newcommand{\lie}[1]{\mathfrak{#1}}

\begin{document}

\title{Automorphisms of ideals of polynomial rings}

\author{Tiago Macedo}
\address{Department of Mathematics and Statistics\\
    University of Ottawa\\
    Ottawa, ON K1N 6N5\\
    \  and \
    Department of Science and Technology\\
    Federal University of S\~ao Paulo\\
    S\~ao Jos\'e dos Campos, S\~ao Paulo, Brazil, 12.247-014}
\thanks{Research of the first author was supported by CNPq grant 232462/2014-3}
\email{tmacedo@unifesp.br}

\author{Thiago Castilho de Mello}
\address{Department of Science and Technology\\
         Federal University of S\~ao Paulo\\
         S\~ao Jos\'e dos Campos, S\~ao Paulo, Brazil, 12.247-014\\}
\thanks{Research of the second author was supported Fapesp grants 2014/10352-4 and 2014/09310-5, and CNPq grants 461820/2014-5 and 462315/2014-2}
\email{tcmello@unifesp.br}

\begin{abstract}
Let $R$ be a commutative integral domain with unit, $f$ be a nonconstant monic polynomial in $R[t]$, and $I_f \subset R[t]$ be the ideal generated by $f$.  In this paper we study the group of $R$-algebra automorphisms of the $R$-algebra without unit $I_f$.  We show that, if $f$ has only one root (possibly with multiplicity), then $\Aut (I_f) \cong R^\times$.  We also show that, under certain mild hypothesis, if $f$ has at least two different roots in the algebraic closure of the quotient field of $R$, then $\Aut(I_f)$ is a cyclic group and its order can be completely determined by analyzing the roots of $f$.
\end{abstract}

\date{\today}

\maketitle

%%%%%%%%%%%%%%%%%%%%%
\section{Introduction}

\subsection{}

Automorphisms of rings and algebras has been a subject of extensive research since the last century, especially automorphisms of the polynomial rings in $n$ variables over a field, $\bbk$.  The most simple type of automorphisms of the algebra $\bbk [x_1, \dotsc, x_n]$ are the so-called elementary automorphisms.  They are defined to be the only homomorphisms of $\bbk [x_1, \dotsc, x_n]$ such that $x_i \mapsto \alpha x_i + f$ for some $i \in \{ 1, \dotsc, n\}$, $\alpha \in \bbk$, $f \in \bbk [x_1,\dotsc, x_{i-1}, x_{i+1}, \dotsc, x_n]$, and $x_j \mapsto x_j$ for all $j \neq i$.  The subgroup of the group of automorphisms of $\bbk [x_1, \dotsc, x_n]$, generated by such automorphisms is called the tame subgroup, and its elements are called tame automorphisms of $\bbk[x_1, \dotsc, x_n]$.  When $\bbk$ is a field of characteristic zero, it was proved by Jung and van der Kulk \cite{Jung42, Kulk53} that every automorphism of $\bbk[x_1, x_2]$ is tame.  Similar results were proved to be true for other classes of free algebras, such as the free associative algebra in two variables \cite{Limanov70, Czerniakiewicz71, Czerniakiewicz72} and finitely generated free Lie algebras \cite{Cohn68}.  It was thus conjectured that every automorphism of $\bbk[x_1, \dotsc, x_n]$ was tame for $n > 2$.  In beginning of the present century Shestakov and Umirbaev \cite{ShestakovUmirbaev, Umirbaev} have shown that this conjecture is false by proving that so-called Nagata automorphism of $\bbk [x_1, x_2, x_3]$ and the so-called Anick automorphism of the free associative algebra in three variables are not tame.

Another important problem related to automorphisms of polynomial rings, is the so-called Jacobian Conjecture (see for example \cite{vandenEssen}).  Given a polynomial map $F \colon \bbk^n \to \bbk^n$, with $F = (F_1, \dotsc, F_n)$ for some $F_i \in \bbk[x_1, \dotsc, x_n]$, denote by $J_F$ its Jacobian matrix.  The Jacobian Conjecture states that, if $\det(J_F) \in \bbk^\times$, then $F$ is an invertible polynomial map; that is, there exists a polynomial map $G \colon \bbk^n \to \bbk^n$, $G = (G_1, \dotsc, G_n)$, such that $x_i = G_i(F_1, \dotsc, F_n)$, for all $i = 1, \dotsc, n$.  A natural approach to this problem is through automorphisms of polynomial rings, as there is a one-to-one correspondence between invertible polynomial maps and $\bbk$-algebra automorphisms of $\bbk[x_1, \dotsc, x_n]$ (given by precomposition).  Several reductions of the Jacobian Conjecture as well as some particular cases have been proven.  Nevertheless, its complete solution is still unknown.

Although automorphisms of rings and algebras have been extensively studied and the group of automorphisms of $\bbk[x]$ is well-known, the authors found no references in the literature regarding the group of automorphisms of ideals of $\bbk [x]$ when considered as $\bbk$-algebras without unit.  The description of such groups may have important applications.  One particular application, and the one that motivated the current paper, is the following.  Let $\lie g$ be a finite-dimensional simple Lie algebra over $\bb C$, and let $\lie g[t]$ denote the current algebra associated to $\lie g$; that is, the Lie algebra with underlying vector space $\lie g \otimes \bb C[t]$ and Lie bracket linearly extending
\[
[x \otimes t^n, y \otimes t^m] = [x,y] \otimes t^{m+n}
\]
for all $x,y \in \lie g$ and $n,m \ge 0$.  Representation theory of the current algebra $\lie g[t]$ is closely related to the representation theory of the affine Kac-Moody Lie algebra $\widehat{\lie g}$ associated to $\lie g$, as $\lie g[t]$ is a parabolic subalgebra of $\widehat{\lie g}$.

For any ideal $I \subset \bb C[t]$, the subspace $\lie g \otimes I$ is an ideal of the Lie algebra $\lie g[t]$.  These ideals can be used to describe extensions between finite-dimensional irreducible $\lie g[t]$-modules.  Namely, Boe, Drupieski, Nakano and the first author proved in \cite{BDMN} that
\[
\Ext^2_{\lie g[t]} (V, V')
\cong \Hom_{\lie g \otimes A/I} (V, \h^2 (\lie g \otimes I, \bb C) \otimes V')
\oplus \h^2 (\lie g \otimes A, \bb C)^{\oplus \delta_{V, V'}},
\]
where $I$ is an ideal of $\bb C[t]$ such that $(\lie g \otimes I) (V^* \otimes V) = 0$, $\delta_{V, V'} = 1$ if $V \cong V'$, and $\delta_{V, V'} = 0$ otherwise.  Notice that this description of $\Ext^2_{\lie g[t]} (V, V')$ depends on the description of $\h^2 (\lie g \otimes I, \bb C)$ as a $\lie g \otimes A/I$-module, which is still unknown.  A technique that could help describe $\h^2 (\lie g \otimes I, \bb C)$ is to use $\Aut(I)$ in order to decompose the cocomplex used for computing $\h^2 (\lie g \otimes I, \bb C)$ into subcocomplexes whose cohomologies may be easier to compute.  A similar approach was used in \cite{vigre7} to compute $\h^2 \left( G(\bb F_q), L \right)$ for a simple, simply-connected algebraic group $G$ defined over $\bb F_p$, $q = p^r$, for a positive prime $p$, $r > 0$, $G(\bb F_q)$ the subgroup of $\bb F_q$-rational points of $G$, and a rational irreducible $G$-module $L$.  The results of the current paper show that this technique in not available for every $\lie g \otimes I$.

%%%%%%%%%%%%%%%%%%%%%%%%%%%%%%%%%%%%%
\subsection{Organization of the paper}

Section~\ref{S:results} contains the main results of this paper.  In Section~\ref{ss:exts} we prove a few general results, which will be repeatedly used in the forthcoming sections, regarding extensions of certain homomorphisms between ideals of the form $I_f$ to endomorphisms of $R[t]$.  In Section~\ref{ss:perms} we describe isomorphisms between ideals of the form $I_f$ in terms of certain bijections between the corresponding sets $Z(f)$.  In particular, automorphisms of $I_f$ are described in terms of certain permutations of $Z(f)$.  In Section~\ref{ss:mains} we prove our main results.  We first describe the automorphism groups of ideals generated by nonconstant monic polynomials with only one root (possibly with multiplicity) in $\bb K$ (see Proposition~\ref{prop:main0}).  Then we proceed to completely describe the automorphism group of $I_f$ for nonconstant monic polynomials $f \in R[t]$ with at least two distinct roots (see Theorems~\ref{thm:main} and \ref{thm:main2}).  We finish the paper by suggesting a few interesting problems in Section~\ref{S:open.problems}.

%%%%%%%%%%%%%%%%%%%%
\subsection{Notation}

Let $R$ be a commutative integral domain with unit, denote by $\bbk$ its quotient field and by $\bb K$ the algebraic closure of $\bbk$.  Also, denote by $R^\times$, $\bbk^\times$ and $\bb K^\times$ the multiplicative groups consisting of invertible elements in $R$, $\bbk$ and $\bb K$ respectively.

Given $f \in R[t]$, we say that $f$ is monic if its leading coefficient is invertible in $R$.  We will denote by $I_f$ the ideal of $R[t]$ generated by $f$ (viewed as an associative commutative $R$-algebra without unit), by $A_f$ the associative commutative $R$-algebra with unit obtained from $I_f$ by formal adjunction of unity; that is, as an $R$-module $A_f = R \oplus I_f$, and we endow it with the product defined by $(r_1 + g_1) (r_2+ g_2) = r_1r_2 + r_1g_2 + r_2g_1 + g_1g_2$ for all $r_1, r_2 \in R$ and $g_1, g_2 \in I_f$.

Given any polynomial $f \in \bb K[t]$, let $Z(f)$ denote its set of roots in $\bb K$, and let $\deg(f)$ denote its degree.  Given a set $S$, let $\sym(S)$ denote the group of permutations of $S$.

%%%%%%%%%%%%%%%%
\section{Results} \label{S:results}

%%%%%%%%%%%%%%%%%%%%%%%%%%%%%%%%%%%%%%%%%%%%%%%%%%%%%%%
\subsection{Extension of isomorphisms and automorphisms} \label{ss:exts}

We begin by proving a few general results related to extensions of homomorphisms between ideals of the form $I_f$ to endomorphisms of $R[t]$.  The main results of this subsection, Propositions~\ref{prop:HW} and \ref{prop:lambda}, will be repeatedly used in the coming sections.  We begin by proving that one can extend any injective algebra homomorphism between ideals of the form $I_f$ to an injective endomorphism of $R[t]$.

\begin{lemma} \label{lem:HW.inj}
Let $f$ and $g$ be polynomials in $R[t]$.  If $f$ is nonconstant and monic, then every injective $R$-algebra homomorphism $\varphi \colon I_f \to I_g$ extends uniquely to an injective $R$-algebra homomorphism $\phi \colon R[t] \to R[t]$.
\end{lemma}

\begin{proof}
We will divide the proof in three steps.  We begin by showing that every injective $R$-algebra homomorphism $\varphi \colon I_f \to I_g$ extends uniquely to an injective $R$-algebra homomorphism $\tilde\varphi \colon A_f \to A_g$.  In fact, every such $\tilde\varphi$ must satisfy $\tilde\varphi (\lambda + h) = \lambda + \varphi (h)$ for all $\lambda \in R$ and $h \in I_f$.  It is easy to verify that $\tilde\varphi$ defined in this way extends $\varphi$ and is an injective $R$-algebra homomorphism.

Now, we will show that every injective $R$-algebra homomorphism $\tilde\varphi \colon A_f \to A_g$ extends uniquely to an injective $R$-algebra homomorphism $\tilde\phi \colon \bbk (t)\to \bbk(t)$.  First notice that $A_f$ and $A_g$ are domains, and $\bbk (t)$ is their quotient field.  Since $\tilde\varphi$ is injective, by composing it with the inclusion of $A_g$ into $\bbk(t)$, we obtain an injective $R$-algebra homomorphism $\tilde\varphi \colon A_f \to \bbk(t)$.  By the universal property of quotient rings (see, for instance, \cite[Proposition 3.1]{atiyah}), there exists a unique ring homomorphism $\tilde\phi \colon \bbk(t) \to \bbk(t)$ extending $\tilde\varphi$.  Since $\tilde\varphi$ is an $R$-algebra homomorphism, $\tilde\phi$ will also be an $R$-algebra homomorphism.

Moreover, since every $\bbk$-algebra endomorphism of $\bbk (t)$  is uniquely determined by the image of $t$, and $\bbk$ is the quotient field of $R$, every $R$-algebra endomorphism of $\bbk (t)$ is uniquely determined by the image of $t$.  Thus $\tilde\phi$ is uniquely define by setting $\tilde\phi(t) = \tilde\varphi (tf) \tilde\varphi(f)^{-1} = \varphi(tf) \varphi(f)^{-1}$.  Finally, notice that $\tilde\phi$ is injective because it is nonzero and $\bbk(t)$ is a field.

Our last step will be to show that the restriction of $\tilde\phi$ to $R[t]$ induces an injective $R$-algebra homomorphism $\phi \colon R[t] \to R[t]$.  First notice that, if we prove that $\tilde\phi (t) \in R[t]$, then $\tilde\phi(h(t)) = h(\tilde\phi(t)) \in R[t]$ for all $h(t) \in R[t]$.  Moreover, since $\tilde\phi$ is an injective $R$-algebra homomorphism that extends $\varphi$, its restriction to $R[t]$ will induce an injective $R$-algebra homomorphism $\phi \colon R[t] \to R[t]$ that extends $\varphi$.

In order to prove that $\tilde\phi (t) \in R[t]$, let $p(t)$ and $q(t)$ be coprime and nonzero polynomials in $R[t]$ such that $\tilde\phi(t) = p(t)q(t)^{-1}$, and let $c_0, c_1, \dotsc, c_{n-1}$ be elements in $R$ such that $f(t) = t^n + c_{n-1}t^{n-1} + \dotsb + c_1 t + c_0$.  Since $\tilde\phi(f(t)) = \tilde\varphi(f(t)) = \varphi(f(t))$ and $\tilde\phi (f(t)) = f(\tilde\phi(t))$, there exists $h(t) \in R[t]$ such that
\[
h(t)= p(t)^n q(t)^{-n} + c_{n-1} p(t)^{n-1} q(t)^{1-n} + \dotsb + c_1 p(t) q(t)^{-1} + c_0 .
\]
Multiplying both sides of this equation by $q(t)^n$ and rearranging the terms, we obtain the following equation in $R[t]$:
\[
-p(t)^n = q(t) \left( (c_0-h(t))q(t)^{n-1} + c_1p(t)q(t)^{n-2} + \dotsb + c_{n-1} p(t)^{n-1} \right).
\]
This equation shows that $q(t)$ divides $p(t)^n$. Since $p(t)$ and $q(t)$ are assumed to be coprime, it follows that $q(t)$ must divide $1$. Thus $q(t) \in R^\times$, proving that $\tilde{\phi}(t)\in R [t]$ and finishing the proof.
\end{proof}

Notice that the extension of the identity map on $I_f$ constructed in the proof of Lemma~\ref{lem:HW.inj} is the identity map on $R[t]$, and that composition of extensions constructed in the proof of Lemma~\ref{lem:HW.inj} is the same as the extension of compositions.  We record this result, as it will be used in the proof of Corollary~\ref{cor:HW} and Lemma~\ref{lem:aut=perm}.

\begin{lemma} \label{lem:comp.exts}
Let $f$, $g$ and $h$ be polynomials in $R[t]$, let $\varphi \colon I_f \to I_g$ and $\psi \colon I_g \to I_h$ be injective homomorphisms of $R$-algebras, and denote $(\psi \circ \varphi)$ by $\xi$.  If $f=g$ is nonconstant and monic and $\varphi = {\rm id}_{I_f}$, then $\varphi$ extends to $\phi = {\rm id}_{R[t]}$.  If $f$ and $g$ are nonconstant and monic, then $\varphi$, $\psi$ and $\xi$ extend to injective $R$-algebra endomorphisms $\tilde\varphi$, $\tilde\psi$ and $\tilde\xi$ of $R[t]$ satisfying $\tilde\psi \circ \tilde\varphi = \tilde\xi$.
\end{lemma}

\begin{proof}
Immediate from the proof of Lemma~\ref{lem:HW.inj}.
\end{proof}

The next result, which is a direct consequence of Lemmas~\ref{lem:HW.inj} and \ref{lem:comp.exts}, shows that any isomorphism between ideals of the form $I_f$ extends uniquely to an automorphism of $R[t]$.

\begin{corollary}[Hahn-Wofsey] \label{cor:HW}
If $f$ and $g$ are nonconstant monic polynomials in $R[t]$, then every $R$-algebra isomorphism $\varphi \colon I_f \to I_g$ extends uniquely to an $R$-algebra automorphism $\phi \colon R[t] \to R[t]$.  Moreover, if $f=g$, then the function $\iota_f \colon \Aut (I_f) \to \Aut (R[t])$ given by $\iota_f (\varphi) = \phi$ is an injective homomorphism of groups.
\end{corollary}

\begin{proof}
By Lemma~\ref{lem:HW.inj}, $\varphi$ extends uniquely to an injective $R$-algebra homomorphism $\tilde\varphi \colon R[t] \to R[t]$.  Now, since $\varphi$ is an isomorphism, it admits an inverse $\psi \colon I_g \to I_f$, which by Lemma~\ref{lem:HW.inj}, also extends uniquely to an injective $R$-algebra homomorphism $\tilde\psi \colon R[t] \to R[t]$.  Since $\varphi$ and $\psi$ are inverses of each other, then $\tilde\varphi$ and $\tilde\psi$ are also inverses of each other by Lemma~\ref{lem:comp.exts}.  This proves the first part of the statement.  The second part follows from the first one and from Lemma~\ref{lem:comp.exts}.
\end{proof}

We finish this subsection by proving necessary and sufficient conditions under which there exists isomorphisms between ideals of the form $I_f$.  They will be repeatedly used in the coming sections.

\begin{proposition} \label{prop:HW}
Let $f$ and $g$ be nonconstant monic polynomials in $R[t]$.  If $\varphi \colon I_f \to I_g$ is an isomorphism of $R$-algebras, then $\deg(f)=\deg(g)$.  Moreover, $\varphi(f) = \lambda g$ for some $\lambda \in R^\times$.
\end{proposition}

\begin{proof}
By Corollary~\ref{cor:HW}, $\varphi$ extends uniquely to an automorphism $\phi$ of $R[t]$.  Since every automorphism of $R[t]$ maps $t$ to $\alpha t + \beta$ for some $\alpha \in R^\times$ and $\beta \in R$, it follows that the degree of $\phi(f) = \varphi(f)$ is the same as the degree of $f$.  Moreover, since $\varphi(f)$ is an element in $I_g$, then $\deg (\varphi(f)) \ge \deg (g)$.  This implies that $\deg (f) \ge \deg (g)$.

Now, since $\varphi$ is assumed to be an isomorphism, there is an inverse isomorphism of $R$-algebras $\varphi^{-1} \colon I_g \to I_f$.  Replacing $\varphi$ by $\varphi^{-1}$ and applying the argument from the previous paragraph to $\varphi^{-1}$, we conclude that $\deg (g) \ge \deg(f)$.  Thus $\deg (f) = \deg (g)$.  Moreover, since $\varphi(f) \in I_g$ and it has the same degree as $g$, then $\varphi(f)$ must be equal to $\lambda g$ for some $\lambda \in R^\times$.
\end{proof}

\begin{proposition} \label{prop:lambda}
Let $f \in R[t]$ be a nonconstant monic polynomial, $\alpha\in R^\times$, $\beta\in R$, and $\phi$ be the unique automorphism of $R[t]$ such that $\phi(t)=\alpha t+\beta$.  The restriction of $\phi$ to $I_f$ is an automorphism of $I_f$ if and only if $\phi(f)=\alpha^{\deg(f)} f$.
\end{proposition}

\begin{proof}
By Proposition~\ref{prop:HW}, if the restriction of $\phi$ to $I_f$ is an automorphism, then $\phi(f)=\lambda f$ for some $\lambda \in R^\times$.  Since $f$ is assumed to be nonconstant and monic, the leading coefficient of $\lambda f (t)$ is $r\lambda$ for some $r \in R^\times$.  Since $\phi(t)$ is assumed to be $\alpha t + \beta$, the leading coefficient of $\phi(f(t)) = f(\phi(t))$ is $r\alpha^{\deg(f)}$.  Since $r \in R^\times$, this proves the \emph{only if} part.

To prove the converse, suppose $\phi(f)=\alpha^{\deg(f)}f$.  Let $\varphi$ denote the restriction of $\phi$ to $I_f$ and notice that $\varphi(I_f)\subseteq I_f$, as
\[
\varphi(h f) = \phi(h f) = \phi(h) \phi(f) = \alpha^{\deg(f)} \phi(h)f
\]
for all $h \in R[t]$.  Moreover, since $\varphi$ is the restriction of an automorphism, it is an injective homomorphism of algebras.  Using a similar argument, one can show that $\varphi^{-1} = \phi^{-1}\arrowvert_{I_f}$ is an injective homomorphism of algebras.  This proves that $\varphi$ is an automorphism of $I_f$.
\end{proof}

%%%%%%%%%%%%%%%%%%%%%%%%%%%%%%%%%%%%%%%%%%%%%%%%%%%%%%%%%%%
\subsection{Relation between automorphisms and permutations} \label{ss:perms}

In this subsection we describe isomorphisms between ideals of the form $I_f$ and $I_g$ in terms of certain bijections between $Z(f)$ and $Z(g)$. In particular, automorphisms of $I_f$ are described in terms of certain permutations of $Z(f)$.

\begin{definition}
Given nonconstant polynomials $f$ and $g \in R[t]$, we define $\Gamma_R (f,g)$ to be the set consisting of all bijections $\eta \colon Z(f) \to Z(g)$ for which:
\begin{enumerate}[(a)]
\item There exists $\alpha \in R^\times$ and $\beta \in R$ such that $\eta(z)=\frac{1}{\alpha}(z-\beta)$ for all $z \in Z(f)$;
\item The multiplicity of $\eta(z)$ as a root of $g$ is the same as the multiplicity of $z$ as a root of $f$.
\end{enumerate}
When $f=g$, we will denote $\Gamma_R (f,f)$ simply by $\Gamma_R (f)$.
\end{definition}

In the particular case when $f=g$ is a nonconstant monic polynomial in $R[t]$, we prove in the next result that $\Gamma_R (f)$ is in fact a group.

\begin{lemma} \label{lem:aut=perm}
For any nonconstant monic polynomial $f \in R[t]$, $\Gamma_R (f)$ is a subgroup of $\sym Z(f)$.
\end{lemma}

\begin{proof}
First notice that the identity element of $\sym Z(f)$ belongs to $\Gamma_R (f)$.  In order to show that $\Gamma_R (f)$ is closed under composition, let $\eta_1,\eta_2\in \Gamma_R (f)$ be such that $\eta_i(z)=\frac{1}{\alpha_i}(z-\beta_i)$. Then:
\begin{align*}
\eta_1\circ\eta_2(z)
& = \eta_1 \left( \frac{1}{\alpha_2}(z-\beta_2) \right) \\
& =\frac{1}{\alpha_1} \left( \frac{1}{\alpha_2} (z-\beta_2)-\beta_1 \right)\\
&=\frac{1}{\alpha_1\alpha_2}(z - (\beta_2+\alpha_2\beta_1)).
\end{align*}
Moreover, since $\eta_1,\eta_2 \in \Gamma_R (f)$, the multiplicity of $z$ is the same as the multiplicity of $\eta_2(z)$ and the same as the multiplicity of $\eta_1(\eta_2(z))$ as a root of $f$.

To finish the proof, notice that the inverse of any $\eta \in \Gamma_R (f)$, $\eta(z)=\frac{1}{\alpha}(z-\beta)$, $\alpha \in R^\times$ and $\beta\in R$, is given by $\eta^{-1}(z)={\alpha}(z+\frac \beta \alpha)$.  Moreover, the multiplicities of $\eta^{-1}(z)$ and $z = \eta (\eta^{-1}(z))$ as roots of $f$ must be the same. Hence $\eta^{-1}$ also belongs to $\Gamma_R (f)$.
\end{proof}

Given two polynomials $f, g \in R[t]$, denote by $\iso_R (I_f, I_g)$ the set of isomorphisms of $R$-algebras from $I_f$ to $I_g$.  The next lemma, which generalizes Lemma~\ref{lem:aut=perm}, gives a bijection between $\iso_R (I_f, I_g)$ and $\Gamma_R (f,g)$.

\begin{lemma} \label{lem:iso=bij}
For any nonconstant monic polynomials $f$ and $g$ in $R[t]$, there exists a bijection between $\iso_R (I_f, I_g)$ and $\Gamma_R(f,g)$.  Moreover, if $f=g$, then this bijection is an isomorphism of groups between $\Aut(I_f)$ and $\Gamma_R (f)$.
\end{lemma}

\begin{proof}
If $\deg(f) \neq \deg(g)$, then $\iso_R (I_f, I_g)$ is empty by Proposition~\ref{prop:HW}, and $\Gamma_R(f,g)$ is empty because there exists no bijection between $Z(f)$ and $Z(g)$ that preserves the multiplicities of every root.  So, we will assume that $\deg(f) = \deg(g)$, and write
\[
f(t) = (t-a_1)^{m_1} \dotsm (t-a_r)^{m_r}
\quad \textup{and} \quad
g(t)=(t-b_1)^{n_1} \dotsm (t-b_r)^{n_r}
\]
for some $r, m_1, n_1, \dotsc, m_r, n_r > 0$ and $a_1, b_1, \dotsc, a_r, b_r \in \bb K$ such that $a_i \neq a_j$ and $b_i\neq b_j$ if $i \neq j$.

Now we will construct, to every $\varphi \in \iso_R (I_f, I_g)$, a corresponding element $\sigma_\varphi \in \Gamma_R(f,g)$. By Corollary~\ref{cor:HW}, every $\varphi \in \iso_R (I_f, I_g)$ extends uniquely to an automorphism $\phi$ of $R [t]$.  Let $\alpha \in R^\times$ and $\beta \in R$ be such that $\phi (t) = \alpha t + \beta$, and observe that
\begin{align*}
h(t) g (t) = \varphi (f(t))
&= \phi (f(t)) \\
&= f (\phi(t)) \\
&= (\phi(t) - a_1)^{m_1} \dotsm (\phi(t) - a_r)^{m_r} \\
&= \alpha^{\deg(f)} \left( t - \frac{a_1 - \beta}{\alpha} \right)^{m_1} \cdots \left( t - \frac{a_r - \beta}{\alpha} \right)^{m_r}
\end{align*}
for some $h \in R[t]$.  Since the degree and the roots of the polynomial on the left are the same as the degree and the roots of the polynomial on the right, it follows that $h \in R^\times$ and for every $i = 1, \dotsc, r$, $\frac{1}{\alpha}(a_i - \beta) = b_j$ for some $j = 1, \dotsc, r$. Denote $b_j$ by $\sigma_{\varphi} (a_i)$ and observe that $\sigma_\varphi \in \Gamma_R (f,g)$.

Now we will construct the inverse of $\varphi \mapsto \sigma_\varphi$.  Given $\sigma \in \Gamma_R (f, g)$, let $\alpha \in R^\times$ and $\beta \in R$ be such that $\sigma (z) = \frac{1}{\alpha} (z - \beta)$ for all $z \in Z(f)$.  Then define $\varphi_\sigma = \phi\arrowvert_{I_f}$ where $\phi$ is the unique automorphism of $R[t]$ satisfying $\phi(t) = \alpha t + \beta$.  Observe that $\varphi_\sigma \in \iso_R (I_f, I_g)$, that $\varphi_{\sigma_\varphi} = \varphi$, and that $\sigma_{\varphi_\sigma} = \sigma$ for all $\varphi \in \iso_R (I_f, I_g)$ and $\sigma \in \Gamma_R (f,g)$.  Thus, we obtained a bijection between $\iso_R (I_f, I_g)$ and $\Gamma_R(f,g)$

To finish the proof, we show that, if $f=g$, then the correspondence $\varphi \mapsto \sigma_\varphi$ is a homomorphism of groups.  Consider $\varphi_1, \varphi_2 \in \Aut(I_f)$ and their extensions $\phi_1 = \iota_f (\varphi_1), \phi_2 = \iota_f (\varphi_2) \in \Aut(R[t])$.  Let $\alpha_1, \alpha_2 \in R^\times$ and $\beta_1, \beta_2 \in R$ be such that $\phi_1(t) = \alpha_1 t + \beta_1$ and $\phi_2 (t) = \alpha_2 t + \beta_2$.  Then, by the above construction, the corresponding permutations of $Z(f)$ are given by $\sigma_{\varphi_1} (z) = \frac{1}{\alpha_1} (z - \beta_1)$ and $\sigma_{\varphi_2} (z) = \frac{1}{\alpha_2} (z - \beta_2)$. Thus
\begin{align*}
\sigma_{\varphi_1} (\sigma_{\varphi_2} (z))
&= \sigma_{\varphi_1} \left( \frac{1}{\alpha_2} (z - \beta_2) \right) \\
&= \frac{1}{\alpha_1} \left( \frac{1}{\alpha_2} (z - \beta_2)- \beta_1 \right)\\
&= \frac{1}{\alpha_1 \alpha_2} (z - (\beta_2 + \alpha_2 \beta_1)).
\end{align*}
Since $\phi_1 (\phi_2 (t)) = \alpha_2 (\alpha_1 t + \beta_1) + \beta_2 = \alpha_1 \alpha_2 t + (\beta_2 + \alpha_2 \beta_1)$, and $\varphi_1 \circ \varphi_2 = \iota_f (\phi_1 \circ \phi_2)$ by Lemma~\ref{lem:comp.exts}, it follows that $\sigma_{\varphi_1} \circ \sigma_{\varphi_2} = \sigma_{\varphi_1 \circ \varphi_2}$ and we conclude that $\Gamma_R(f)$ and $\Aut(I_f)$ are isomorphic.
\end{proof}

Notice that, since $Z(f)$ is a finite set, then $\sym Z(f)$ is a finite group.  Thus, from Lemmas~\ref{lem:aut=perm} and \ref{lem:iso=bij}, we see that $\Aut(I_f)$ is a finite group too.  In Theorem~\ref{thm:main} we will show that $\Aut(I_f)$ is also cyclic, thus abelian.

The next result follows from Lemma~\ref{lem:iso=bij} and will be used in the proof of Theorem~\ref{thm:main2}.

\begin{lemma}
Let $f\in R[t]$ be a nonconstant monic polynomial.  Then for all $k>0$,
\[
\Aut(I_f) \cong \Aut(I_{f^k}).
\]
\end{lemma}

\begin{proof}
By Lemma~\ref{lem:iso=bij}, $\Aut (I_f) \cong \Gamma_R (f)$ and $\Aut (I_{f^k}) \cong \Gamma_R (f^k)$.  Since $Z(f) = Z(f^k)$ and the multiplicity of each root of $f^k$ is $k$ times its multiplicity as a root of $f$, then $\Gamma_R (f^k) \cong \Gamma_R (f)$.
\end{proof}

The next result, which also follows directly from Lemma~\ref{lem:iso=bij}, gives a more concrete criterion to determine if two ideals $I_f$ and $I_g$ are isomorphic.

\begin{corollary} \label{cor:If=Ig}
Let $f$ and $g$ be nonconstant monic polynomials in $R[t]$.  The ideals $I_f$ and $I_g$ are isomorphic if and only if $\Gamma_R (f,g)$ is nonempty. \qed
\end{corollary}

Given $z_0 \in R$ and $f \in R[t]$, denote the polynomial $f(t-z_0) \in R[t]$ by $f_{z_0}(t)$.  The next result shows that such a shift in the roots of $f$ does not change the isomorphism class of the group $\Gamma_R (f)$.

\begin{lemma} \label{lem:perm.shift}
For any $z_0\in R$ and any nonconstant monic polynomial $f \in R[t]$, there exists an isomorphism of groups between $\Gamma_R (f)$ and $\Gamma_R (f_{z_0})$.
\end{lemma}

\begin{proof}
First, notice that, for any $z_0 \in R$, $Z (f_{z_0}) = \{ z + z_0 \mid z \in Z(f) \}$.  Now, given a permutation $\sigma \colon Z(f) \to Z(f)$, consider the permutation $\sigma_{z_0} \colon Z(f_{z_0}) \to Z(f_{z_0})$ given by $\sigma_{z_0} (z) = \sigma(z - z_0) + z_0$. Observe that, if $\sigma$ is a permutation of $Z(f)$ which satisfies $\sigma (z) = \frac 1 \alpha (z - \beta)$ for some $\alpha \in R^\times$ and $\beta \in R$, then $\sigma_{z_0}$ is a permutation of $Z(f_{z_0})$ which satisfies $\sigma_{z_0} (z) = \frac{1}{\alpha} (z - \beta')$ with $\beta' = ((1-\alpha)z_0 + \beta) \in R$.  Also observe that, if $\sigma$ is a permutation of $Z(f)$ such that the multiplicity of $\sigma (z)$ as a root of $f$ is the same as that of $z$, then $\sigma_{z_0}$ is a permutation of $Z(f_{z_0})$ such that the multiplicity of $\sigma_{z_0} (z + z_0) = \sigma (z) + z_0$ is the same as that of $z + z_0$.

To finish the proof, we only need to prove that the correspondence $\sigma \mapsto \sigma_{z_0}$ is a homomorphism of groups, since $(\sigma_{z_0})_{-z_0} (z) = \sigma_{z_0} (z+z_0)-z_0 = (\sigma(z)+z_0)-z_0 = \sigma(z)$ for all $z \in Z(f)$.  Given $\sigma, \rho \in \sym Z(f)$, we have
\begin{align*}
(\sigma \circ \rho)_{z_0} (z)
&= \sigma (\rho (z - z_0)) + z_0 \\
&= \sigma ((\rho (z - z_0) + z_0) - z_0) + z_0 \\
&= \sigma (\rho_{z_0} (z) - z_0) + z_0 \\
&= \sigma_{z_0} (\rho_{z_0} (z))
\end{align*}
for all $z \in Z(f)$.  This shows that the correspondence $\sigma \mapsto \sigma_{z_0}$ is a homomorphism of groups.
\end{proof}

Let $f$ be a nonconstant monic polynomial in $R[t]$, $\varphi \in \Aut(I_f)$, and let $\phi$ denote its extension $\iota_f (\varphi) \in \Aut (R[t])$ (see Corollary~\ref{cor:HW}).  Recall that there exist unique $\alpha \in R^\times$ and $\beta \in R$ such that $\phi (t) = \alpha t + \beta$.  Now, let $z_0 \in R$ and $\phi_{z_0}$ be the unique automorphism of $R[t]$ that satisfies $\phi_{z_0}(t) = \alpha t + ((1-\alpha)z_0 + \beta)$.  Define a function $\zeta_0 \colon \Aut(I_f) \to \Aut (I_{f_{z_0}})$ by setting $\zeta_0 (\varphi) = \phi_{z_0} \arrowvert_{I_{f_{z_0}}}$.  We finish this subsection by using Lemmas~\ref{lem:iso=bij} and \ref{lem:perm.shift} to show that $\zeta_0$ is in fact an isomorphism of groups.  This result will be important in the proofs of Lemma~\ref{lem:beta=0} and Theorem~\ref{thm:main}.

\begin{proposition} \label{prop:aut.shift}
For any $z_0 \in R$ and any nonconstant monic polynomial $f \in R[t]$, the function $\zeta_0 \colon \Aut(I_f) \to \Aut (I_{f_{z_0}})$ is an isomorphism of groups.
\end{proposition}

\begin{proof}
By Lemma~\ref{lem:iso=bij}, there exists an isomorphism of groups $\Aut (I_f) \to\Gamma_R (f)$.  Namely, for each $\varphi \in \Aut(I_f)$, there exist unique $\alpha \in R^\times$ and $\beta \in R$ such that $\left( \iota_f (\varphi) \right) (t) = \alpha t + \beta$.  The isomorphism $\Aut (I_f) \to\Gamma_R (f)$ assigns to $\varphi \in \Aut (I_f)$ the permutation $\sigma_\varphi \in \Gamma_R (f)$ given by $\sigma_\varphi (z) = \frac{1}{\alpha} (z - \beta)$ for all $z \in Z(f)$.

By Lemma~\ref{lem:perm.shift}, there exists an isomorphism of groups $\Gamma_R (f) \to \Gamma_R (f_{z_0})$ which assigns to a permutation $\sigma \in \Gamma_R (f)$ the permutation $\sigma_{z_0} \in \Gamma_R (f_{z_0})$ given by $\sigma_{z_0} (z) = \frac{1}{\alpha} (z - ((1-\alpha)z_0 + \beta))$ for all $z \in Z(f_{z_0})$.

Since $f_{z_0} \in R[t]$, by Lemma~\ref{lem:iso=bij}, there exists an isomorphism of groups $\Gamma_R (f_{z_0}) \to \Aut (I_{f_{z_0}})$.  Namely, for each permutation $\sigma \in \Gamma_R (f_{z_0})$, there exist unique $\alpha \in R^\times$ and $\beta \in R$ such that $\sigma (z) = \frac{1}{\alpha} (z - \beta)$ for all $z \in Z (f_{z_0})$.  The isomorphism $\Gamma_R (f_{z_0}) \to \Aut (I_{f_{z_0}})$ assigns to $\sigma$ the automorphism $\varphi_\sigma = \psi \arrowvert_{I_{f_{z_0}}} \in \Aut (I_{f_{z_0}})$, where $\psi$ is the unique automorphism of $R[t]$ such that $\psi(t) = \alpha t + \beta$.

Composing these three isomorphisms of groups, we obtain an isomorphism $\Aut(I_f) \to \Aut (I_{f_{z_0}})$ that assigns to each $\varphi \in \Aut(I_f)$ the automorphism $\zeta_0 (\varphi) \in \Aut(I_{f_{z_0}})$.
\end{proof}

%%%%%%%%%%%%%%%%%%%%%%%%
\subsection{Main results} \label{ss:mains}

In this subsection we will prove the main results of the paper:  we completely describe the automorphism group of $I_f$ for any nonconstant monic polynomial $f \in R[t]$.  We separate our analysis in two main cases and begin by describing the automorphism groups of ideals generated by polynomials with only one root in $\bb K$.

\begin{proposition} \label{prop:main0}
If $f(t) = (t-a)^n$ for some $a \in R$ and $n \ge 1$, then there exists an isomorphism of groups between $\Aut (I_f)$ and $R^\times$.
\end{proposition}

\begin{proof}
Suppose $\varphi \in \Aut(I_f)$, denote $\iota_f(\varphi) \in \Aut (R[t])$ by $\phi$, and let $\alpha \in R^\times$ and $\beta \in R$ be such that $\phi(t) = \alpha t + \beta$.  By Proposition~\ref{prop:HW}, there exists $\lambda \in \bbk^\times$ such that
\[
\lambda (t - a)^n
= \phi (f(t))
= (\alpha t + (\beta - a))^n
= \alpha^n \left( t - \frac{a - \beta}{\alpha} \right)^n.
\]
Since the roots of the polynomials on the left and right sides of this equation are the same, $\beta$ must be equal to $(1-\alpha)a$.

Now, given $\alpha \in R^\times$, let $\phi_\alpha$ be the unique automorphism of $R[t]$ satisfying $\phi_\alpha (t) = \alpha t + (1-\alpha)a$.  Notice that
\[
\phi_\alpha (f(t))
= (\alpha t + (1-\alpha)a - a )^n
= \alpha^n (t-a)^n
= \alpha^n f(t).
\]
Thus, $\phi_\alpha \arrowvert_{I_f}$ is an automorphism of $I_f$.  This shows that there exists a bijection between $R^\times$ and $\Aut (I_f)$, explicitly given by $\alpha \mapsto \phi_\alpha \arrowvert_{I_f}$.  To finish the proof, we will show that this bijection is in fact an isomorphism of groups.  For any $\alpha_1, \alpha_2 \in R^\times$, we have
\begin{align*}
\phi_{\alpha_1} (\phi_{\alpha_2} (t))
&= \phi_{\alpha_1} (\alpha_2 t + (1 - \alpha_2) a) \\
&= \alpha_2 (\alpha_1 t + (1 - \alpha_1) a ) + (1 - \alpha_2) a \\
&= \alpha_1 \alpha_2 t + (1 - \alpha_1 \alpha_2) a \\
&= \phi_{\alpha_1 \alpha_2} (t).
\end{align*}
This implies that $\phi_{\alpha_1} \arrowvert_{I_f} \circ \phi_{\alpha_2} \arrowvert_{I_f} = \phi_{\alpha_1 \alpha_2} \arrowvert_{I_f}$, and finishes the proof.
\end{proof}

For the rest of this section, we will analyze the case of polynomials with at least two distinct roots in $\bb K$.  In the next result, we use Proposition~\ref{prop:HW} to obtain restrictions on the coefficients $\alpha$ and $\beta$ that describe the extensions to $R[t]$ of automorphisms of $I_f$ obtained in Corollary~\ref{cor:HW}.

\begin{lemma} \label{lem:beta=0}
Let $f \in R [t]$ be decomposed as $f(t) = (t-a_1)^{m_1} \dotsm (t-a_r)^{m_r}$, with $r, m_1, \dotsc, m_r > 0$ and $a_1, \dotsc, a_r \in \bb K$ distinct. Let $n = m_1 + \dotsb + m_r > 1$ denote the degree of $f$, let $\varphi \in \Aut(I_f)$, let $\phi$ denote $\iota_f (\varphi) \in \Aut (R[t])$, and let $\alpha \in R^\times$ and $\beta \in R$ be such that $\phi(t) = \alpha t + \beta$.  Then
\[
n \beta = (1-\alpha)(m_1a_1 + \dotsb + m_ra_r).
\]
Moreover, if $f$ has at least two distinct roots in $\bb K$, $n$ is nonzero in $R$, and there exists $z_0 \in R$ such that $n z_0 = -(m_1a_1 + \dotsb + m_ra_r)$, then $\alpha$ is an $(n-m_0)$-th root of unity, where $m_0$ is the multiplicity of $-z_0$ as a root of $f$.
\end{lemma}

\begin{proof}
From Proposition~\ref{prop:HW}, we know that there exists $\lambda \in R^\times$ such that $\varphi(f) = \lambda f$.  Since $\varphi$ is a restriction of $\phi$, we have on the one hand
\begin{align} \label{eq:lambdaf1}
\varphi(f(t))
&= (\alpha t + (\beta-a_1))^{m_1} \dotsm (\alpha t + (\beta-a_r))^{m_r} \nonumber \\
&= (\alpha t)^n + (n \beta-(m_1a_1+\dotsb+m_ra_r)) (\alpha t)^{n-1} + \dotsb + f(\beta),
\end{align}
and on the other hand
\begin{equation} \label{eq:lambdaf2}
\lambda f(t)
= \lambda t^n - \lambda (m_1a_1 + \dotsb + m_ra_r) t^{n-1} + \dotsb + (-1)^n\lambda a_1^{m_1} \dotsm a_r^{m_r}.
\end{equation}
By comparing the terms on the right sides of equations~\eqref{eq:lambdaf1} and \eqref{eq:lambdaf2}, we obtain that
\begin{equation} \label{eq:rels1}
\alpha^n = \lambda
\quad \textup{and} \quad
n\beta = (1-\alpha)(m_1a_1+\dotsb+m_ra_r).
\end{equation}
This proves the first part of the statement.

To prove the second part, let $z_0 \in R$ be such that $n z_0 = -(m_1a_1 + \dotsb + m_ra_r)$, and consider $f_{z_0} \in R[t]$.  Let $\psi \in \Aut(I_{f_{z_0}})$, let $\tilde\psi$ denote $\iota_f (\psi) \in \Aut (R[t])$, and let $\tilde\alpha \in R^\times$ and $\tilde\beta \in R$ be the unique elements such that $\tilde\psi(t) = \tilde\alpha t + \tilde\beta$.  Notice that by Proposition~\ref{prop:aut.shift}, $\tilde{\alpha}=\alpha$.  Also notice that $f_{z_0}$ is decomposed as $f_{z_0}(t) = (t-(a_1+z_0))^{m_1} \dotsm (t-(a_r+z_0))^{m_r}$, with $r>1$, $m_1, \dotsc, m_r > 0$, $(a_1+z_0), \dotsc, (a_r+z_0) \in \bb K$ distinct, and $\deg (f_{z_0}) = n$.  As in the first part of the proof, there exists $\tilde\lambda \in R^\times$ such that
\begin{equation} \label{eq:rel'}
\psi(f_{z_0})= \tilde\lambda f_{z_0},
\quad
\alpha^n = \tilde\lambda
\quad \textup{and} \quad
n\tilde\beta = (1-\alpha)(m_1(a_1+z_0) + \dotsb + m_r(a_r+z_0)).
\end{equation}
Together with equation~\eqref{eq:rels1}, equation~\eqref{eq:rel'} implies, in particular, that $\tilde\lambda = \lambda$.  Moreover, notice that $m_1(a_1+z_0) + \dotsb + m_r(a_r+z_0) = (m_1a_1 + \dotsb + m_ra_r) + n z_0 = 0$.  Since $R$ is assumed to be a domain and $n$ is assumed to be nonzero in $R$, equation~\eqref{eq:rel'} implies that $\tilde\beta = 0$.

Denote by $m_0 \ge 0$ the multiplicity of $0$ as a root of $f_{z_0}$ (or, equivalently, the multiplicity of $-z_0$ as a root of $f$).  Since $f$ is assumed to have at least two distinct roots in $\bb K$, then $f_{z_0}$ has at least one nonzero root in $\bb K$.  Hence, equation~\eqref{eq:lambdaf1} for $f_{z_0}$ becomes
\begin{equation} \label{eq:lambdaf3}
\psi(f_{z_0}(t))
= (\alpha t)^n + \dotsb + (-1)^{n-m_0} \prod_{a_i + z_0 \neq 0} (a_i + z_0)^{m_i} (\alpha t)^{m_0},
\end{equation}
and equation~\eqref{eq:lambdaf2} becomes
\begin{equation} \label{eq:lambdaf4}
\lambda f_{z_0}(t)
= \lambda t^n + \dotsb + (-1)^{n-m_0} \lambda \prod_{a_i + z_0 \neq 0} (a_i + z_0)^{m_i} t^{m_0}.
\end{equation}
Since $R$ is assumed to be a domain, the coefficients of $t^{m_0}$ on equations~\eqref{eq:lambdaf3} and \eqref{eq:lambdaf4} are nonzero.  By comparing them, we obtain that $\alpha^{m_0} = \lambda = \alpha^n$.  Since $\alpha \in R^\times$, then $\alpha^{n-m_0}=1$, that is, $\alpha$ is an $(n-m_0) $-th root of unity.
\end{proof}

Using Proposition~\ref{prop:aut.shift} and Lemma~\ref{lem:beta=0}, we are now able to show that $\Aut (I_f)$ is a finite cyclic group (thus abelian) when $f$ has at least two distinct roots in $\bb K$.

\begin{theorem} \label{thm:main}
Let $f$ be a nonconstant monic polynomial in $R[t]$, decomposed as
\[
f(t) = (t-a_1)^{m_1} \dotsm (t-a_r)^{m_r},
\]
with $r > 1$, $m_1, \dotsc, m_r > 0$ and $a_1, \dotsc, a_r$ being its distinct roots in $\bb K$. Let $n=m_1+\cdots+m_r>1$ denote the degree of $f$.  If there exists $z_0 \in R$ such that $n z_0 = (m_1a_1 + \dotsb + m_ra_r)$, then $\Aut(I_f)$ is isomorphic to a cyclic subgroup of $\Aut(R[t])$ of order $(n-m_0)$, where $m_0 \ge 0$ denotes the multiplicity of $-z_0$ as a root of $f$.  In particular, $\Aut(I_f)$ is cyclic.
\end{theorem}

\begin{proof}
If there exists $z_0 \in R$ such that $n z_0 = (m_1a_1 + \dotsb + m_ra_r)$, then by Proposition~\ref{prop:aut.shift}, $\Aut (I_f)$ is isomorphic to $\Aut (I_{f_{z_0}})$.  So without loss of generality, we will assume that $z_0 = 0$.  Since $f$ has at least two different roots in $\bb K$, at least one of these roots is nonzero.  Then by Lemma~\ref{lem:beta=0}, for any $\varphi \in \Aut(I_f)$, there exists an $(n-m_0)$-th root of unity, $\alpha_\varphi$, such that $\varphi = \phi \arrowvert_{I_f}$ where $\phi \in \Aut (R[t])$ satisfies $\phi(t)=\alpha_\varphi t$.

We will show that the function $\varphi \mapsto \alpha_\varphi$ is an injective homomorphism of groups.  If $\alpha_{\varphi_1} = \alpha_{\varphi_2}$, then $\varphi_1 = \iota_f (\phi_1) \arrowvert_{I_f}$, $\varphi_2 = \iota_f (\phi_2) \arrowvert_{I_f}$ and $\iota_f \left( \varphi_1 \right) (t) = \alpha_{\varphi_1} t = \alpha_{\varphi_2} t = \iota_f \left( \varphi_2 \right) (t)$.  This proves that the function $\varphi \mapsto \alpha_\varphi$ is injective.  Now, if $\varphi, \psi \in \Aut(I_f)$, then $\alpha_{\varphi\circ\psi} t = \varphi \circ \psi (t) = \varphi (\alpha_\psi t) = \alpha_\varphi \alpha_\psi t$. Hence $\alpha_{\varphi\circ\psi}=\alpha_\varphi \alpha_\psi$.  This shows that $\varphi \mapsto \alpha_\varphi$ is homomorphism of groups.  Moreover, by Lemma~\ref{lem:beta=0}, $\Aut(I_f)$ is isomorphic to a subgroup of the cyclic group of order $(n-m_0)$.
\end{proof}

\begin{remark}
In the case where $R$ is a field, if the characteristic of $R$ is zero, or if the characteristic of $R$ is positive and does not divide the degree of $f \in R[t]$, then the result of Theorem~\ref{thm:main} is valid for every polynomial $f$ with at least two distinct roots in $\bb K$.  However, the result of Theorem~\ref{thm:main} may not remain valid when the characteristic of $R$ divides the degree of $f$; that is, when $\deg (f)$ is zero in $R$.  For instance, if $R = \bb Z / 3 \bb Z$ and $f (t) = t^3 - t$, then one can check that $\Aut (I_f) \cong \Aut (R[t])$, which is isomorphic to the permutation group in three letters, thus not a cyclic group.
\end{remark}

One can visualize the result of Theorem~\ref{thm:main} in the following way.  Let $R = \bb R$ be the field of real numbers, and $f$ be a nonconstant polynomial in $\bb R [t]$ with two or more distinct roots, possibly in $\bb K = \bb C$, the field of complex numbers.  Then each nontrivial automorphism of $I_f$ is induced by a rotation of the complex plane that fixes the point $z_0$, which is the average of the roots of $f$ (counted with multiplicities).

For each nonconstant monic polynomial $f \in R[t]$ and each $m > 0$, denote by $Z_m (f)$ the subset of $Z(f)$ consisting of roots of $f$ with multiplicity $m$ and by $f_m(t)$ the polynomial in $\bb K [t]$ defined by
\[
f_m(t)=
\begin{cases}
\prod_{a \in Z_m(f)} (t-a), & \textup{if $Z_m (f) \ne \emptyset$}; \\
1, & \textup{otherwise}.
\end{cases}
\]
Notice that $Z_m (f) = \emptyset$ for all $m > \deg(f)$, and that $f(t) = \prod_{m>0} (f_m(t))^m$.

Notice that, if $m>0$ is such that $Z_m(f) = \emptyset$, then $I_{f_m} = R[t]$.  Recall from Corollary~\ref{cor:HW} that, for any nonconstant monic polynomial $f \in R[t]$, there exists an injective homomorphism of groups $\iota_f \colon \Aut (I_f) \inj \Aut (R[t])$.  For each $m > 0$, let $\iota_m \colon \Aut(I_{f_m}) \inj \Aut(R[t])$ be defined by
\[
\iota_m (\varphi) =
\begin{cases}
\iota_{f_m}(\varphi), & \textup{if $Z_m (f) \ne \emptyset$}; \\
{\rm id}_{R[t]}, & \textup{otherwise}.
\end{cases}
\]

In the next result we will establish a relation between the automorphism groups of $I_{f_m}$ and the automorphism group of $I_f$.

\begin{theorem} \label{thm:main2}
Let $f \in R[t]$ be a nonconstant monic polynomial.  Then $\iota_f$ induces an isomorphism of groups between $\Aut(I_f)$ and $\bigcap_{m>0} \iota_m \left( \Aut(I_{f_m}) \right)$.
\end{theorem}

\begin{proof}
We will first show that $\iota_f (\varphi) \in \bigcap_{m>0} \iota_m \left( \Aut(I_{f_m}) \right)$ for all $\varphi \in \Aut(I_f)$.  Let $\phi \in \Aut(R[t])$ denote $\iota_f(\varphi)$ and recall that there exist unique elements $\alpha \in R^\times$ and $\beta \in R$ such that $\phi(t) = \alpha t + \beta$.  It is clear that $\phi \in \iota_m \left( \Aut (I_{f_m}) \right)$ for all $m > 0$ such that $Z_m (f) = \emptyset$.  So, assume that $Z_m (f) \ne \emptyset$ and notice that, since $f_m$ is monic and nonconstant, the leading coefficient of $\phi(f_m)$ is $r_m \alpha ^{\deg(f_m)}$, where $r_m \in R^\times$ is the leading coefficient of $f_m$.  By Lemma~\ref{lem:aut=perm}, the roots and multiplicities of the roots of $\phi(f)$ are the same as those of $f$.  Thus $\phi(f)_m = \phi(f_m) = \alpha^{\deg(f_m)} f_m$.  By Proposition~\ref{prop:lambda}, this implies that $\phi\arrowvert_{I_{f_m}} \in \Aut (I_{f_m})$.  Now, by the uniqueness of extensions proved in Lemma~\ref{lem:HW.inj}, we obtain that $\iota_m (\phi\arrowvert_{I_{f_m}}) = \phi$.

To finish the proof, we will show that, if $\phi \in \iota_m \left( \Aut(I_{f_m}) \right)$ for all $m > 0$, then $\phi \arrowvert_{I_f} \in \Aut (I_f)$.  Let $\alpha \in R^\times$ and $\beta \in R$ be such that $\phi(t) = \alpha t + \beta$.  By the uniqueness of extensions proved in Lemma~\ref{lem:HW.inj}, we have that $\phi \arrowvert_{I_{f_m}} \in \Aut (I_{f_m})$ for all $m > 0$.  Notice that, if $Z_m (f) = \emptyset$, then $f_m = 1$ and $\phi (f_m) = \alpha^{\deg (f_m)} f_m$.  If $Z_m (f) \ne \emptyset$, then $f_m$ is a nonconstant monic polynomial, and by Proposition~\ref{prop:lambda}, $\phi(f_m)= \alpha^{\deg(f_m)} f_m$.   Thus, we have:
\[
\phi(f)
=\phi \left( \prod_{m>0} f_m^m \right)
=\prod_{m>0} \phi(f_m)^m
=\prod_{m>0} \alpha^{m \deg(f_m)}f_m^m
=\alpha^{\sum m \deg(f_m)} \prod_{m>0} f_m^m
=\alpha^{\deg(f)} f.
\]
Hence, by Proposition~\ref{prop:lambda}, $\phi \arrowvert_{I_f} \in \Aut(I_f)$.
\end{proof}

Together with Theorem~\ref{thm:main}, Theorem~\ref{thm:main2} gives a very precise description of the group $\Aut (I_f)$.  In fact, if $f$ is a nonconstant monic polynomial with at least two distinct roots in $\bb K$, then by Theorem~\ref{thm:main}, $\Aut(I_f)$ is a finite, cyclic (thus abelian) group.  We finish this section by using Theorem~\ref{thm:main2} to pin down the order of $\Aut(I_f)$.

\begin{corollary}
Let $f$ be a nonconstant monic polynomial in $R[t]$ and, for each $m > 0$, let $d_m$ denote the order of the group $\Aut(I_{f_m})$.  Then $\Aut(I_f)$ is isomorphic to a cyclic group of order $\gcd(\{d_m \colon Z_m \neq \emptyset\})$.  In particular, if $\gcd(\{d_m \colon Z_m \neq \emptyset\})=1$, then $\Aut(I_f)$ is trivial. \qed
\end{corollary}

%%%%%%%%%%%%%%%%%%%%%%
\section{Open problems} \label{S:open.problems}

\subsection{Isomorphisms of ideals of polynomial rings}

Let $n, k, \ell > 0$ and $f_1, \dotsc, f_k, g_1, \dotsc, g_\ell$ be polynomials in $R [x_1, \dotsc, x_n]$.  Find conditions on $f_1, \dotsc, f_k, g_1, \dotsc, g_\ell$ under which the ideals $\langle f_1, \dotsc, f_k \rangle$ and $\langle g_1, \dotsc, g_\ell \rangle$ are isomorphic as associative commutative algebras without unit.  Notice that, for $n=1$, this problem is solved in Corollary~\ref{cor:If=Ig} and Proposition~\ref{prop:lambda}.

\subsection{Varieties determined by one variable polynomials}

Let $f \in R[t]$ be a nonconstant monic polynomial and let $n>0$ denote $\deg(f)$. Notice that $I_f$ and $A_f$ are finitely generated $R$-algebras. The former is generated by $\{f, tf, \dotsc, t^{n-1}f\}$ and the latter is generated by $\{1, f, tf, \dotsc, t^{n-1}f\}$.  Denote by $\mathfrak p$ the maximal ideal of $R [x_1, \dotsc, x_n]$ generated by $x_1, \dotsc, x_n$, and define the following homomorphisms of $R$-algebras:
\[\begin{array}{lcr}
\begin{array}{cccc}
\theta_f \colon & R [x_1, \dotsc, x_n] & \longrightarrow &  A_f\\
       &        x_i          & \longmapsto     & t^{i-1}f\\
       &         1           & \longmapsto     & 1     \\
\end{array}
& \text{ and } &
\begin{array}{cccc}
\theta_f^\circ \colon & \mathfrak p & \longrightarrow &  I_f\\
&        x_i          & \longmapsto     & t^{i-1}f\\
\end{array}\\
\end{array}  \]
Since these homomorphisms are surjective,
\[
\begin{array}{lcr}
A_f \cong \dfrac{R[x_1, \dotsc, x_n]}{\ker \theta_f}
& \text{ and } &
I_f \cong \dfrac{\mathfrak p}{\ker \theta_f^\circ}.
\end{array}
\]

The ideal $\ker \theta_f$ in $R [x_1, \dotsc, x_n]$ corresponds to an affine variety $\mathcal{V} (\ker \theta_f)$ in $\bb A^n$, whose coordinate ring is $A_f$.  We denote $\mathcal V(\ker \theta_f)$ by $\mathcal V_f$ and call it \emph{the variety determined by $f$}.

In the set of all varieties determined by one variable polynomials, one can consider their isomorphism classes as affine varieties.  It is well known that two affine varieties are isomorphic if and only if their coordinate rings are isomorphic.  Since any isomorphism between $I_f$ and $I_g$ can be extended to an isomorphism between $A_f$ and $A_g$ (see, for instance, the first paragraph of the proof of Lemma~\ref{lem:HW.inj}), we have the following result:

\begin{proposition} \label{isom:var}
Let $f$ and $g$ be nonconstant monic polynomials in $R[t]$.  If $I_f$ is isomorphic to $I_g$, then $\mathcal{V}_f$ is isomorphic to $\mathcal{V}_g$.
\end{proposition}

One can consider the following two questions.  Is the converse of Proposition~\ref{isom:var} true?  For each fixed $n>0$, are there finitely many isomorphism classes of affine varieties determined by one variable polynomials?

%%%%%%%%%%%%%%%%%%%%%%%%%%%%
\subsection*{Acknowledgments}

The first author would like to thank CNPq and the second author would like to thank CNPq and Fapesp for their financial supports.  The first author would also like to thank D. Nakano for bringing up the problem that motivated this paper, and both authors would like to thank J. Hahn and E. Wofsey for their important suggestions \cite{HW}.

%%%%%%%%%%%%%%%%%

\end{document}